\documentclass[10pt,a4paper,reqno]{amsart}

\usepackage[utf8]{inputenc}
\usepackage[T1]{fontenc}
\usepackage{lmodern}

\usepackage[english]{babel}

\usepackage{amsmath}
\usepackage{amsfonts}
\usepackage{amssymb}
\usepackage{amsthm}

\usepackage{hyperref}

\theoremstyle{plain}
	\newtheorem{theorem}{Theorem}
	\newtheorem{proposition}[theorem]{Proposition}
	
	\newtheorem{corollary}[theorem]{Corollary}
\theoremstyle{definition}
	\newtheorem{definition}[theorem]{Definition}
	
	\newtheorem{example}[theorem]{Example}
\theoremstyle{remark}
	\newtheorem{remark}[theorem]{Remark}

\DeclareMathOperator{\alg}{\mathrm{alg}}
\DeclareMathOperator{\rad}{\mathrm{rad}}
\DeclareMathOperator{\Arf}{\mathrm{Arf}}

\newcommand{\Cl}{\mathrm{Cl}}
\newcommand{\bmu}{\boldsymbol{\mu}}
\newcommand{\bydef}{:=}

\begin{document}

\title{Clifford algebras as twisted group algebras and the Arf invariant}

\author{Alberto Elduque}
\address{Departamento de Matem\'aticas e
Instituto Uni\-ver\-si\-ta\-rio de Matem\'aticas y Apli\-ca\-cio\-nes,
Universidad de Zaragoza, 50009 Zaragoza, Spain.}
\email{elduque@unizar.es}

\author{Adri\'an Rodrigo-Escudero}
\address{Departamento de Matem\'aticas e
Instituto Uni\-ver\-si\-ta\-rio de Matem\'aticas y Apli\-ca\-cio\-nes,
Universidad de Zaragoza, 50009 Zaragoza, Spain.}
\email{adrian.rodrigo.escudero@gmail.com}

\date{22 January 2018}

\begin{abstract}
Some connections between quadratic forms over the field of two elements,
Clifford algebras of quadratic forms over the real numbers,
real graded division algebras,
and twisted group algebras
will be highlighted.
This allows to revisit real Clifford algebras in terms of
the Arf invariant of the associated quadratic forms over the field of two elements,
and give new proofs of some classical results.
\end{abstract}

\maketitle

\section{Introduction}

While studying graded-division finite-dimensional simple real algebras \cite{Rod16},
the second author found that they are intimately related
to regular quadratic forms over the field of two elements $\mathbb{F}_2$,
and their Arf invariant.
A similar connection was remarked by Ovsienko \cite{Ovs16} between
real Clifford algebras and regular quadratic forms over $\mathbb{F}_2$.

Actually, the graded-division algebras
with homogeneous components of dimension $1$
are nothing else but the twisted group algebras,
dealt with by many authors in different situations,
and the expression of Clifford algebras as twisted group algebras
was given by Albuquerque and Majid \cite{AM02}.
Here the groups behind the Clifford algebras
are the $2$-elementary abelian groups.

Let us mention that
Ovsienko also studied with Morier-Genoud \cite{MO10}
the real (and complex) Clifford algebras
as graded-commutative algebras,
that is, as graded algebras such that
the commutation relations of the homogeneous elements
are given by a bicharacter.
They proved that every
finite-dimensional simple (associative) graded-commutative algebra
over the field of real or complex numbers
is isomorphic to a Clifford algebra,
and that the bicharacter can be chosen to be
($-1$ raised to the power of)
the usual scalar product on $\mathbb{Z}_2^N$.

\medskip

The goal of this paper is to link the results of
Rodrigo-Escudero \cite{Rod16},
Ovsienko \cite{Ovs16},
and Albuquerque and Majid \cite{AM02},
connecting real graded-division algebras,
twisted group algebras, Clifford algebras,
and quadratic forms over $\mathbb{F}_2$.
Thanks to these connections,
we show that we can determine
the isomorphism class of a real Clifford algebra
in terms of the Arf invariant,
which allows us to give new proofs
(proofs of Theorem \ref{th:Arf_mupq} and Corollary \ref{cor:periodicity})
of some classical results.

\medskip

After reviewing the basic definitions on gradings
and on graded division algebras in Section \ref{se:gradings},
twisted group algebras will be considered in Section \ref{se:twisted},
where Theorem \ref{th:twisted}
relating alternating bicharacters and twisted group algebras
will be proved.
In Section \ref{se:Clifford_twisted},
the Clifford algebras of the nondegenerate quadratic forms
will be shown to be isomorphic to some twisted group algebras.
This was considered first in \cite{AM02}.
Section \ref{se:qf} will be devoted to review
the classical classification by Dickson \cite{Dic01}
of the regular quadratic forms over $\mathbb{F}_2$.
They are determined by their Arf invariant.

In Section \ref{se:qf_algebras}
each quadratic form over $\mathbb{F}_2$ is shown to define uniquely
a real twisted group algebra over a $2$-elementary abelian group,
the isomorphism class of this algebra is determined by
its dimension and the corresponding Arf invariant (Corollary \ref{co:isos}).

Finally, nice formulas for the Arf invariant
are given in Section \ref{se:real_Clifford}
(Theorem \ref{th:Arf_mupq}),
and the previous connections between real Clifford algebras,
twisted group algebras, and quadratic forms over $\mathbb{F}_2$
are used here to reprove some well-known results on real Clifford algebras:
the determination of their isomorphism class
(Corollary \ref{cor:Clifford_isos}),
and some periodicity results
(Corollary \ref{cor:periodicity}).

\section{Background on gradings}\label{se:gradings}

We want to study the relation between
Clifford algebras and twisted group algebras,
which are naturally graded algebras,
so let us recall the basic notions about gradings,
following \cite{EK13}.
Let $G$ be a group,
which will be assumed to be abelian in most cases,
a $G$-\emph{grading} on an $\mathbb{F}$-vector space $W$
is a decomposition of $W$ into
a direct sum of subspaces indexed by $G$,
\begin{equation}
\Gamma : W = \bigoplus_{ g \in G } W_g .
\end{equation}
The \emph{support} of $\Gamma$
(or of the $G$-graded vector space $W$)
is the subset of the grading group
$ \mathrm{supp} (\Gamma) := \{ g \in G \mid W_g \neq 0 \} $.
If $ 0 \neq  w \in W_g $,
we say that $w$ is \emph{homogeneous} of degree $g$,
and write $ \deg w = g $.
The subspace $W_g$ is called the
\emph{homogeneous component} of degree $g$.
A subspace $U$ of $W$ is said to be \emph{graded}
if $ U = \bigoplus_{ g \in G } ( W_g \cap U ) $.
Given two gradings on the same vector space,
$ \Gamma : W = \bigoplus_{ g \in G } W_g $ and
$ \Gamma' : W = \bigoplus_{ h \in H } W'_h $,
we say that $\Gamma'$ is a \emph{coarsening} of $\Gamma$,
or that $\Gamma$ is a \emph{refinement} of $\Gamma'$,
if for any $ g \in G $ there exists $ h \in H $
such that $ W_g \subseteq W'_h $.

If a $G$-grading $\Gamma$ on an algebra $\mathcal{D}$ also satisfies
\begin{equation}
\mathcal{D}_g \mathcal{D}_h \subseteq \mathcal{D}_{gh}
\end{equation}
for all $ g,h \in G$,
we say that $\Gamma$ is a \emph{group grading};
we will always assume that
the gradings on our algebras are group gradings.
The graded algebra $\mathcal{D}$ is said to be
a \emph{graded division} algebra if
the left and right multiplications
by any nonzero homogeneous element
are invertible operators.
If $\mathcal{D}$ is associative this is equivalent to $\mathcal{D}$ being unital
and every nonzero homogeneous element having an inverse.
In that case $ 1 \in \mathcal{D}_e $,
where $e$ is the neutral element of $G$,
and if $ 0 \neq X \in \mathcal{D}_g $,
then $ X^{-1} \in \mathcal{D}_{g^{-1}} $;
so the support of $\mathcal{D}$ is a subgroup of $G$.
Indeed, whenever $ \mathcal{D}_g \neq 0 $
and $ \mathcal{D}_h \neq 0 $, we also have
$ 0 \neq \mathcal{D}_g \mathcal{D}_h \subseteq \mathcal{D}_{gh} $
and $ \mathcal{D}_{g^{-1}} \neq 0 $.

Two group gradings
$ \Gamma  : \mathcal{D} = \bigoplus_{ g \in G } \mathcal{D}_g $ and
$ \Gamma' : \mathcal{E} = \bigoplus_{ h \in H } \mathcal{E}_h $ are
\emph{equivalent} if there is an isomorphism of algebras
$ \varphi : \mathcal{D} \rightarrow \mathcal{E} $,
such that for any $ g \in \mathrm{supp} (\Gamma) $,
there is an $ h \in \mathrm{supp} (\Gamma') $
such that $ \varphi (\mathcal{D}_g) = \mathcal{E}_h $.
If $H=G$ and $ \varphi (\mathcal{D}_g) = \mathcal{E}_g $ for any $ g \in G $,
then $\varphi$ is said to be an \emph{isomorphism} of $G$-graded algebras.

\section{Twisted group algebras}\label{se:twisted}

Given a group $G$, a field $\mathbb{F}$ and
a map $ \sigma : G \times G \rightarrow \mathbb{F}^{\times} $,
the \emph{twisted group algebra} $ \mathbb{F}^{\sigma} G $ is
the algebra over $\mathbb{F}$ with a basis consisting of a copy of $G$:
$ \{ \varepsilon_g : g \in G \} $,
and with (bilinear) multiplication given by
\begin{equation}
\varepsilon_g \varepsilon_h \bydef \sigma(g,h) \varepsilon_{gh}
\end{equation}
for any $ g,h \in G $.

In a natural way, $ \mathbb{F}^{\sigma} G $ is a $G$-graded algebra:
$ \left( \mathbb{F}^{\sigma} G \right)_g \bydef \mathbb{F} \varepsilon_g $,
which is not necessarily associative.
It is a graded-division algebra because
the left and right multiplications by $\varepsilon_g$
are invertible operators,
as $\sigma$ takes values in $\mathbb{F}^{\times}$.
Actually,
\emph{any $G$-graded-division algebra (not necessarily associative),
with homogeneous components of dimension $1$,
is a twisted group algebra},
isomorphic to $ \mathbb{F}^{\sigma} T $,
for a suitable $\sigma$, where $T$ is the support of the grading.

\begin{example}\label{ex:octonions}
With $ \mathbb{F} = \mathbb{R} $,
Albuquerque and Majid (see \cite{AM99}) considered
the classical algebras of complex numbers,
quaternions and octonions as the
twisted group algebras $ \mathbb{R}^{\sigma} T $ with:
\begin{itemize}
\item $T=\mathbb{Z}_2$ and
$ \sigma(x,y) = (-1)^{xy} $ for the complex numbers.
\item $T=\mathbb{Z}_2^2$ and
$ \sigma \bigl( (x_1,x_2),(y_1,y_2) \bigr) = (-1)^{x_1y_1+(x_1+x_2)y_2} $
for the real associative division algebra $\mathbb{H}$ of quaternions.
\item $T=\mathbb{Z}_2^3$ and
\[
\sigma \bigl( (x_1,x_2,x_3),(y_1,y_2,y_3) \bigr)
= (-1)^{ y_1x_2x_3 + x_1y_2x_3 + x_1x_2y_3
+ \sum_{ 1 \leq i \leq j \leq 3 } x_iy_j }
\] 
for the real non-associative division algebra $\mathbb{O}$ of octonions.
\end{itemize}
\end{example}

The twisted group algebra $ \mathbb{F}^{\sigma} G $
is associative if, and only if, for any $ g,h,k \in G $,
$ ( \varepsilon_g \varepsilon_h ) \varepsilon_k =
\varepsilon_g ( \varepsilon_h \varepsilon_k ) $,
and this is equivalent to the following condition on $\sigma$:
\begin{equation}\label{eq:cocycle}
\sigma(g,h) \sigma(gh,k) = \sigma(h,k) \sigma(g,hk) ,
\end{equation}
that is, to $\sigma$ being a $2$-cocycle
with values in $\mathbb{F}^{\times}$:
$ \sigma \in Z^2(G,\mathbb{F}^{\times}) $.
(See \cite[Subsection 1.2]{Pas77} for the basic properties
of associative twisted group algebras.)

Two twisted group algebras
$ \mathbb{F}^{\sigma} G $ and $ \mathbb{F}^{\sigma'} G $
are isomorphic as $G$-graded algebras
(these are named `diagonally equivalent' in \cite{Pas77})
if and only if there is a map
$ \mu : G \rightarrow \mathbb{F}^{\times} $
such that the linear map determined by
$ \varepsilon_g \mapsto \mu(g) \varepsilon_g' $,
where the $\varepsilon_g'$'s are the elements
of the natural basis of $ \mathbb{F}^{\sigma'} G $,
is an isomorphism of algebras, and this happens if, and only if,
\begin{equation}
\mu(gh) \sigma(g,h) = \mu(g) \mu(h) \sigma'(g,h) ,
\end{equation}
that is, if and only if $\sigma$ and $\sigma'$ are cohomologous.

Note that for $ \sigma \in Z^2(G,\mathbb{F}^{\times}) $,
with $h=k=e$ and with $g=h=e$ (the neutral element of $G$)
in Equation \eqref{eq:cocycle},
we get $\sigma(e,e)=\sigma(g,e)=\sigma(e,g)$ for any $ g \in G $,
and hence the element $ 1 = \sigma(e,e)^{-1} \varepsilon_e $
is the unity element of $ \mathbb{F}^\sigma G $.
Moreover, the powers of the basic elements are given by
$ \varepsilon_g^n =
\left( \prod_{i=1}^{n-1} \sigma(g^i,g) \right) \varepsilon_{g^n} $,
and hence, if the order of $g$ is $n$ we have
\begin{equation}\label{eq:powers}
\varepsilon_g^n =
\left( \prod_{i=0}^{n-1} \sigma(g^i,g) \right) 1 .
\end{equation}

\medskip

Given a $2$-cocycle
$ \sigma \in Z^2(G,\mathbb{F}^{\times}) $,
consider the map $ \beta : G \times G \rightarrow \mathbb{F}^\times $
given by $ \beta(g,h) = \sigma(g,h) \sigma(h,g)^{-1} $.
In other words, $\beta$ is determined by the condition:
\begin{equation}\label{eq:beta_comm}
\varepsilon_g \varepsilon_h
= \beta(g,h) \varepsilon_h \varepsilon_g
\end{equation}
for any $ g,h \in G $, that is,
$\beta$ measures the lack of commutativity of $ \mathbb{F}^\sigma G $.
Trivially we have $\beta(g,g)=1$ for any $ g \in G $.

The associativity of $ \mathbb{F}^\sigma G $ gives, for any $ g,h,k \in G $:
\[
\varepsilon_g \varepsilon_h \varepsilon_k = \begin{cases}
	\sigma(g,h) \varepsilon_{gh} \varepsilon_k =
	\sigma(g,h) \beta(gh,k) \varepsilon_k \varepsilon_{gh} ,
	\\
	\beta(g,k) \beta(h,k) \varepsilon_k \varepsilon_g \varepsilon_h =
	\sigma(g,h) \beta(g,k) \beta(h,k) \varepsilon_k \varepsilon_{gh} ,
\end{cases}
\]
and therefore $\beta$ is multiplicative in the first variable,
$ \beta(gh,k) = \beta(g,k) \beta(h,k) $,
and likewise in the second;
so $\beta$ is an \emph{alternating}
($\beta(g,g)=1$ for any $ g \in G $)
\emph{bicharacter} on $G$.

\medskip

Our next result shows that,
assuming that $G$ is a finitely generated abelian group,
any alternating bicharacter on $G$ comes from
an associative twisted group algebra defined on $G$.

\begin{theorem}\label{th:twisted}
Let $G$ be a finitely generated abelian group,
and let $ \beta : G \times G \rightarrow \mathbb{F}^\times $
be an alternating bicharacter on $G$,
then there exists a $2$-cocycle
$ \sigma \in Z^2(G,\mathbb{F}^\times) $ such that
$ \beta(g,h) = \sigma(g,h) \sigma(h,g)^{-1} $.

Moreover, if $ G = \langle g_1 \rangle
\times \cdots \times \langle g_N \rangle $,
for a finite number of elements $ g_1 , \ldots , g_N $,
with $g_i$ of order $ m_i \in \mathbb{N}_{ \geq 2 } $
for $ i = 1,\ldots,r $,
and $g_i$ of infinite order for $ i = r+1,\ldots,N $;
and if $ \mu_1,\ldots,\mu_r \in \mathbb{F}^\times $ are chosen arbitrarily,
then the $2$-cocycle $\sigma$ above can be taken so that
the twisted group algebra $ \mathbb{F}^\sigma G $ is isomorphic,
as a $G$-graded algebra, to the following
unital associative algebra defined by generators and relations:
\begin{multline}\label{eq:Ammubeta}
\alg \big\langle
x_1,\ldots,x_N
\mid x_i^{m_i} = \mu_i \text{, } i = 1,\ldots,r ; \\
x_i x_j = \beta(g_i,g_j) x_j x_i \text{, } i,j = 1,\ldots,N
\big\rangle ,
\end{multline}
which is a $G$-graded algebra with
$\deg(x_i)=g_i$ for $i=1,\ldots,N$.
\end{theorem}

Note that the free associative algebra generated by $N$ generators
can be graded by any group $G$ by assigning arbitrarily degrees
to the free generators.
Hence the grading in the algebra
in Equation \eqref{eq:Ammubeta} is well defined,
as the relations imposed on the generators are homogeneous.

For future reference, the algebra in Equation \eqref{eq:Ammubeta},
which depends on $N$,
the $m_i$'s,
the $\mu_i$'s,
and the alternating bicharacter $\beta$,
will be denoted by
$ \mathcal{A}_{\mathbb{F}} (N,\underline{m},\underline{\mu},\beta) $,
where $ \underline{m} = (m_1,\ldots,m_r) $ and
$ \underline{\mu} = (\mu_1,\ldots,\mu_r) $,
$ 0 \leq r \leq N $.
($ \underline{m} = \emptyset = \underline{\mu} $ if $r=0$.)

\begin{proof}
Define $ \sigma : G \times G \rightarrow \mathbb{F}^\times $ by
\begin{equation}\label{eq:sigma}
\sigma \left( g_1^{a_1} \cdots g_N^{a_N} , g_1^{b_1} \cdots g_N^{b_N} \right)
= \left( \prod_{i>j} \beta \left( g_i^{a_i} , g_j^{b_j} \right) \right)
\left( \prod_{i=1}^r \mu_i^{ \epsilon_i(a_i,b_i) } \right) ,
\end{equation}
for $ a_i , b_i \in \mathbb{Z} $ for $i=1,\ldots,N$,
with $ 0 \leq a_i,b_i < m_i $ for $i=1,\ldots,r$,
and where, for $i=1,\ldots,r$,
$ \epsilon_i (a_i,b_i) = 1 $ if $ a_i+b_i \geq m_i $,
and $ \epsilon_i (a_i,b_i) = 0 $ otherwise.

It is clear that $ \sigma(g,h) \sigma(h,g)^{-1} = \beta(g,h) $
for any $ g,h \in G $.
We must check that $\sigma$ thus defined is a $2$-cocycle,
and this is done by induction on $N$.
If $N=1$ and $r=0$,
then $ \sigma ( g_1^{a_1} , g_1^{b_1} ) = 1 $ for any $a_1,b_1$,
and $\sigma$ is the trivial cocycle.
On the other hand, if $N=r=1$,
$ \sigma ( g_1^{a_1} , g_1^{b_1} ) = \mu_1^{ \epsilon_1(a_1,b_1) } $,
$ 0 \leq a_1,b_1 < m_1 $.
For $x=g_1^{a_1}$, $y=g_1^{b_1}$, and $z=g_1^{c_1}$,
with $ 0 \leq a_1,b_1,c_1 < m_1 $,
\[
\sigma(x,y) \sigma(xy,z) = \sigma(y,z) \sigma(x,yz) = \begin{cases}
	1 & \text{if } a_1+b_1+c_1 < m_1 , \\
	\mu_1 & \text{if } m_1 \leq a_1+b_1+c_1 < 2 m_1 , \\
	\mu_1^2 & \text{if } 2 m_1 \leq a_1+b_1+c_1 .
\end{cases}
\]
Now assume that $N>1$ and that the result is valid for $N-1$.
For $ x = g_1^{a_1} \cdots g_N^{a_N} $ write
$ x' = g_1^{a_1} \cdots g_{N-1}^{a_{N-1}} $,
so that $ x = x' g_N^{a_N} $,
and similarly for $ y = g_1^{b_1} \cdots g_N^{b_N} $
and $ z = g_1^{c_1} \cdots g_N^{c_N} $.
Then
\[
\sigma(x,y) = \sigma(x',y')
\beta ( g_N^{a_N} , y' )
\sigma( g_N^{a_N} , g_N^{b_N} )
.
\]
The factor $\sigma(x',y')$ is a $2$-cocycle (case $N-1$),
and so is the factor $ \sigma( g_N^{a_N} , g_N^{b_N} ) $ (case $N=1$).
On the other hand,
with $ \tilde{\sigma}(x,y) \bydef \beta (g_N^{a_N},y') $,
we have
\[
\begin{split}
\tilde\sigma(y,z) \tilde\sigma(x,yz) &
= \beta (g_N^{b_N},z') \beta (g_N^{a_N},y'z')
= \beta (g_N^{b_N},z') \beta (g_N^{a_N},y') \beta (g_N^{a_N},z') , \\
\tilde\sigma(x,y) \tilde\sigma(xy,z) &
= \beta (g_N^{a_N},y') \beta ( g_N^{a_n} g_N^{b_N} , z' )
= \beta (g_N^{a_N},y') \beta (g_N^{a_N},z') \beta(g_N^{b_N},z') ,
\end{split}
\]
and this shows that $\tilde{\sigma}(x,y)$ is a $2$-cocycle too.

Finally, in the twisted group algebra $ \mathbb{F}^\sigma G $ we have,
because of Equation \eqref{eq:powers}, that for $i=1,\ldots,r$:
\[
\varepsilon_{g_i}^{m_i}
= \left( \prod_{j=0}^{m_i-1} \sigma(g_i^j,g_i) \right) 1
= \mu_i 1 .
\]
Hence the generators $\varepsilon_{g_i}$, $i=1,\ldots,N$,
of $ \mathbb{F}^\sigma G $ satisfy the relations
defining the algebra in Equation \eqref{eq:Ammubeta},
so there is a surjective homomorphism of $G$-graded algebras
$ \varphi : \mathcal{A}_{\mathbb{F}}
(N,\underline{m},\underline{\mu},\beta)
\rightarrow \mathbb{F}^\sigma G $
that sends $x_i$ to $\varepsilon_{g_i}$.
But for any $ g = g_1^{a_1} \cdots g_N^{a_n} \in G $,
the homogeneous component $ \mathcal{A}_{\mathbb{F}}
(N,\underline{m},\underline{\mu},\beta)_g $
is spanned by the monomial $ x_1^{a_1} \cdots x_N^{a_N} $,
so its dimension is at most $1$.
It follows that this dimension is exactly $1$
and that $\varphi$ is an isomorphism.
\end{proof}

\section{Clifford algebras are twisted group algebras}\label{se:Clifford_twisted}

We start recalling the definition of Clifford algebra.
Let $\mathbb{F}$ be a field of characteristic different from $2$,
let $V$ be an $\mathbb{F}$-vector space of finite dimension $N$,
and let $ Q : V \to \mathbb{F} $ be a non-degenerate quadratic form on $V$.
The fact that $Q$ is a quadratic form is equivalent to the existence of a
a symmetric bilinear form $ B : V \times V \to \mathbb{F} $ such that
\begin{equation}
2 Q(v) = B(v,v)
\end{equation}
for all $ v \in V $.
We can recover $B$ from $Q$, since
\begin{equation}\label{eq:bf_from_qf}
B(u,v) = Q(u+v)-Q(u)-Q(v)
\end{equation}
for all $ u,v \in V $.
Denote by $T(V)$ the tensor algebra of $V$,
and by $I(Q)$ the two-sided ideal of $T(V)$
generated by all elements of the form
$ v \otimes v - Q(v)1 $ with $ v \in V $.
The \emph{Clifford algebra} of $(V,Q)$
is $ \Cl(V,Q) := T(V) / I(Q) $.

Let $ \{ v_1 , \ldots , v_N \} $ be an orthogonal basis of $V$,
then by its own definition $\Cl(V,Q)$ is isomorphic to the algebra
\begin{equation}\label{eq:Cl_generators_relations}
\alg \big\langle x_1,\ldots,x_N \mid
x_i^2=\mu_i \text{, } i=1,\ldots,N \text{; }
x_ix_j=-x_jx_i \text{, } i \neq j \big\rangle ,
\end{equation}
under an isomorphism that takes each generator $x_i$
to the class of $v_i$ modulo $I(Q)$.
That is, $\Cl(V,Q)$ is isomorphic to the algebra
$ \mathcal{A}_{\mathbb{F}} (N,\underline{2},\underline{\mu},\beta) $
in Equation \eqref{eq:Ammubeta}, where:
\begin{itemize}
\item $ \underline{2} = (2,2,\ldots,2) $ ($N$ components),
\item $ \underline{\mu} = \bigl( Q(v_1),Q(v_2),\ldots,Q(v_N) \bigr) $,
\item $\beta$ is the alternating bicharacter on
the cartesian product $G=C_2^N$ of
$N$ copies of the cyclic group of two elements
(so that $ G = \langle g_1 \rangle \times \cdots \times \langle g_N \rangle $
with $g_i$ or order $2$ for any $i=1,\ldots,N$),
with $ \beta (g_i,g_j) = -1 $ for any $ i \neq j $.
\end{itemize}

If we also have $ Q(v_1) = \dots = Q(v_N) = 1 $,
we denote the Clifford algebra by $\Cl_N(\mathbb{F})$.
In the case $ \mathbb{F} = \mathbb{R} $,
if $ Q(v_1) = \dots = Q(v_p) = +1 $ and
$ Q(v_{p+1}) = \dots = Q(v_N) = -1 $,
we denote the Clifford algebra by $\Cl_{p,q}(\mathbb{R})$,
where $ p+q = N $.

Theorem \ref{th:twisted} tells us that
$ \mathcal{A}_{\mathbb{F}} (N,\underline{2},\underline{\mu},\beta) $
is graded-isomorphic to a twisted group algebra $ \mathbb{F}^\sigma G $
and, in particular, its dimension is $2^N$.
It also follows that $\Cl(V,Q)$ is a $G$-graded-division algebra,
with $\deg(v_i)=g_i$ for $i=1,\ldots,N$
(here we identify $ v_i \in V $ with
the class of $v_i$ modulo $I(Q)$ in $\Cl(V,Q)$),
which is the grading by the group $ C_2^N \cong \mathbb{Z}_2^N $
that Albuquerque and Majid constructed in \cite[Proposition 2.2]{AM02}.

\begin{remark}
Write $ v_I := v_{i_1} \cdots v_{i_k} $,
where $ 1 \leq i_1 < \dots < i_k \leq N $
and $ I = \{ i_1 , \dots , i_k \} $,
and also denote $ v_{\emptyset} := 1 $.
It is clear that $\Cl(V,Q)$ is linearly spanned by
$ \{v_I\}_{ I \subseteq \{ 1,\dots,N \} } $,
but the fact that this family is linearly independent,
that is, $ \dim \Cl(V,Q) = 2^N $,
is not immediate
(see for example \cite[II.1.2]{Che54} or
\cite[Theorem V.1.8]{Lam05}).
In our case, this is a consequence of Theorem \ref{th:twisted}.
Alternatively, we can use the natural
$\mathbb{Z}_2^N$-grading of Albuquerque and Majid
to give a direct proof:
Since the $v_I$'s belong to different homogeneous components,
they are linearly independent as long as they are nonzero;
and they are nonzero because they are invertible,
since $ Q(v_i) \neq 0 $ for $i=1,\dots,N$
(the quadratic form $Q$ is non-degenerate).
Of course, for this argument to be complete,
it has to be proved first that $ 1 \not\in I(Q) $,
for example embedding $\Cl(V,Q)$ in
the algebra of endomorphisms of the exterior algebra $\Lambda(V)$,
as in \cite[page 38]{Che54}.
Finally note that this also implies that $ \dim \Cl(V,Q) = 2^N $
in the general case, that is, if the quadratic form $Q$ is degenerate,
since we can embed $\Cl(V,Q)$ in
$ \mathbb{F}[X_1] / (X_1^2-Q(v_1)) \otimes \dots \otimes
\mathbb{F}[X_N] / (X_N^2-Q(v_N)) \otimes \Cl_N(\mathbb{F}) $,
and we have proved that $ \dim \Cl_N(\mathbb{F}) = 2^N $.
\end{remark}

\begin{remark}
There are division gradings on Clifford algebras
with grading group different from $ C_2^N \cong \mathbb{Z}_2^N $.
For instance, we can construct an example
of a division grading on $\Cl_3(\mathbb{F})$
by the group $ \mathbb{Z}_2 \times \mathbb{Z}_4
= \langle a \rangle \times \langle b \rangle $
following \cite[Subsection 2.4]{BZ16}:
\begin{equation}
\begin{array}{ll}
	e = \deg 1 , \quad
	& a = \deg v_1 ,
	\\ b = \deg ( v_2 + v_3 v_1 ) , \quad
	& ab = \deg ( v_3 - v_1 v_2 ) ,
	\\ b^2 = \deg v_1 v_2 v_3 , \quad
	& ab^2 = \deg v_2 v_3 ,
	\\ b^3 = \deg ( v_2 - v_3 v_1 ) , \quad
	& ab^3 = \deg ( v_3 + v_1 v_2 ) .
\end{array}
\end{equation}
However, in these graded algebras
the information of the Clifford structure is lost,
and it is more natural to study them
focusing only on their graded nature
(see \cite{BZ16,Rod16,BKR17}),
so we are not interested in them.
\end{remark}

\section{Quadratic forms over the field of two elements}\label{se:qf}

Here we will consider quadratic forms
$ \mathbf{q} : W \rightarrow \mathbb{F}_2 $
defined on a finite dimensional vector space $W$
over the field $\mathbb{F}_2$ of two elements.
Let $ b_{\mathbf{q}} : W \times W \rightarrow \mathbb{F}_2 $
be the associated bilinear form as in Equation \eqref{eq:bf_from_qf}.
As $2=0$ in $\mathbb{F}_2$, $b_{\mathbf{q}}$ is alternating:
$ b_{\mathbf{q}} (w,w) = 0 $ for any $ w\in W $.

Two quadratic forms $ \mathbf{q} : W \rightarrow \mathbb{F}_2 $ and
$ \mathbf{q}' : W' \rightarrow \mathbb{F}_2 $ are \emph{equivalent}
(and we will write $ \mathbf{q} \sim \mathbf{q}' $)
if there is a linear isomorphism $ \varphi : W \rightarrow W' $ such that
$ \mathbf{q}' \bigl( \varphi(w) \bigr) = \mathbf{q}(w) $ for any $ w \in W $.

The \emph{orthogonal sum} $ \mathbf{q} \perp \mathbf{q}' $
is the quadratic form on $ W \times W' $ given by
\begin{equation}
( \mathbf{q} \perp \mathbf{q}' )(w,w')
= \mathbf{q}(w) + \mathbf{q}'(w')
\end{equation}
for any $ w \in W $ and $ w' \in W' $.

\medskip

Dickson \cite{Dic01}
(see also \cite{Ovs16} and \cite[Section 5]{BKR17})
classified long ago the quadratic forms $\mathbf{q}$
on $\mathbb{F}_2^N$ with $ \dim \rad (b_{\mathbf{q}}) \leq 1 $.
(The radical of the bilinear form $b_{\mathbf{q}}$ is the subspace
$ \{ \mathbf{x} \in \mathbb{F}_2^N \mid
b_{\mathbf{q}} (\mathbf{x},\mathbf{y}) = 0 \text{, }
\forall \mathbf{y} \in \mathbb{F}_2^N \} $.)
Dickson's results are summarized as follows: Any quadratic form
$ \mathbf{q} : \mathbb{F}_2^N \rightarrow \mathbb{F}_2 $
with $ \dim \rad (b_{\mathbf{q}}) \leq 1 $
is equivalent to one of the following:
\begin{description}
\item[$N$ even] ($N=2n$)
\begin{itemize}
	\item $ \mathbf{q}_0^N (x_{1},\ldots,x_{2n}) =
	x_{1}x_{2} + x_{3}x_{4} + \cdots + x_{2n-1}x_{2n} $,
	\item $ \mathbf{q}_1^N (x_{1},\ldots,x_{2n}) =
	x_{1}x_{2} + x_{3}x_{4} + \cdots + x_{2n-1}x_{2n}
	+ x_1 + x_2 $,
	(note that $x=x^2$ for any $ x \in \mathbb{F}_2 = \{0,1\} $).
\end{itemize}
\item[$N$ odd] ($N=2n+1$)
\begin{itemize}
	\item $ \mathbf{q}_0^N (x_{1},\ldots,x_{2n+1}) =
	x_{1}x_{2} + x_{3}x_{4} + \cdots + x_{2n-1}x_{2n} $,
	\item $ \mathbf{q}_1^N (x_{1},\ldots,x_{2n+1}) =
	x_{1}x_{2} + x_{3}x_{4} + \cdots + x_{2n-1}x_{2n}
	+ x_1 + x_2 $,
	\item $ \mathbf{q}_2^N (x_{1},\ldots,x_{2n+1}) =
	x_{1}x_{2} + x_{3}x_{4} + \cdots + x_{2n-1}x_{2n}
	+ x_{2n+1} $.
\end{itemize}
\end{description}

\begin{remark}
For $N$ even $ \rad( b_{\mathbf{q}_0^N} )
= 0 = \rad( b_{\mathbf{q}_1^N} ) $,
while for $N$ odd $ \dim \rad( b_{\mathbf{q}_i^N} ) = 1 $ for $i=0,1,2$,
and only $\mathbf{q}_2^N$ is \emph{regular}
(this means that either $ \rad( b_{\mathbf{q}} ) = 0 $,
or $ \dim \rad( b_{\mathbf{q}} ) = 1 $ but $\mathbf{q}$
is not trivial on $ \rad( b_{\mathbf{q}} ) $).
\end{remark}

The equivalence class of a quadratic form $\mathbf{q}$
on $\mathbb{F}_2^N$ with $ \dim \rad( b_{\mathbf{q}} ) \leq 1 $
is determined by its \emph{Arf invariant} $\Arf(\mathbf{q})$,
which tells us whether there are more elements $ w \in \mathbb{F}_2^N $
with $\mathbf{q}(w)=0$ or $\mathbf{q}(w)=1$:

\begin{definition}
Let $ \mathbf{q} : W \rightarrow \mathbb{F}_2 $ be a quadratic form
on a finite dimensional vector space $W$ over $\mathbb{F}_2$.
Then
\begin{equation}
\Arf(\mathbf{q}) \bydef
\begin{cases}
1 & \text{if } \lvert \{ w \in W \mid \mathbf{q}(w) = 0 \} \rvert
> \lvert \{ w \in W \mid \mathbf{q}(w) = 1 \} \rvert , \\
0 & \text{if } \lvert \{ w \in W \mid \mathbf{q}(w) = 0 \} \rvert
= \lvert \{ w \in W \mid \mathbf{q}(w) = 1 \} \rvert , \\
-1 & \text{if } \lvert \{ w \in W \mid \mathbf{q}(w) = 0 \} \rvert
< \lvert \{ w \in W \mid \mathbf{q}(w) = 1 \} \rvert .
\end{cases}
\end{equation}
\end{definition}

An easy computation shows:
\begin{equation}\label{eq:arf_q0q1q2}
\Arf(\mathbf{q}_0^N) = 1 , \qquad 
\Arf(\mathbf{q}_1^N) =-1 , \qquad 
\Arf(\mathbf{q}_2^N) = 0
\text{ (} N \text{ odd).}
\end{equation}

\medskip

A quadratic form that will be important later on
is the one in the next example:

\begin{example}\label{ex:mu_pq}
Take $ p,q \in \mathbb{N} \cup \{0\} $ such that $p+q=N$.
The quadratic form $ \mathbf{q}_{p,q} :
\mathbb{F}_2^N \rightarrow \mathbb{F}_2 $ is defined by:
\begin{equation}
\mathbf{q}_{p,q} (x_1,\ldots,x_N) =
x_{q+1}^2 + \cdots + x_N^2 +
\sum_{ 1 \leq i < j \leq N } x_i x_j .
\end{equation}
Remark that the associated alternating bilinear form is:
\begin{equation}
b_{\mathbf{q}_{p,q}}
\left( ( x_1 , \dots , x_N ) , ( y_1 , \dots , y_N ) \right)
= \sum_{ 1 \leq i \neq j \leq N } x_i y_j .
\end{equation}
Denote by $ \mathbf{e}_1 = (1,0,0,\ldots,0) $,
$ \mathbf{e}_2 = (0,1,0,\ldots,0) $, \ldots,
$ \mathbf{e}_N = (0,0,0,\ldots,1) $,
the canonical generators of $\mathbb{F}_2^N$.
Clearly $ b_{ \mathbf{q}_{p,q} } ( \mathbf{e}_i , \mathbf{e}_j ) = 1 $
for any $ i \neq j $.
Note that $ b_{ \mathbf{q}_{p,q} } $ depends only on $N=p+q$,
and not on $p$ and $q$,
and it is easily checked that for $N$ even
$ b_{ \mathbf{q}_{p,q} } $ is nondegenerate,
while for $N$ odd its radical has dimension $1$.
\end{example}

For later use, it will be convenient to move to multiplicative notation.
Let $ C_2 = \{ \pm 1 \} $ be
the cyclic group of order $2$ `in multiplicative version'.
A multiplicative quadratic form on $C_2^N$ is
a map $ \bmu : C_2^N \rightarrow C_2 $ such that
the map $ \beta_{\bmu} : C_2^N \times C_2^N \rightarrow C_2 $ given by
\begin{equation}
\beta_{\bmu}(r,s) = \mu(rs) \mu(r) \mu(s)
\end{equation}
is an alternating
(that is, $ \beta_{\bmu}(r,r) = 1 $ for any $ r \in C_2^N $)
bicharacter.
Note that $ \bmu(r) = \bmu(r)^{-1} $ for any $r\in C_2^N$.

The multiplicative versions of the quadratic forms
$\mathbf{q}_i^N$ above will be denoted by $\bmu_i^N$,
and the multiplicative version of $\mathbf{q}_{p,q}$
in Example \ref{ex:mu_pq} will be denoted by $\bmu_{p,q}$.

\section{From quadratic forms to real algebras}\label{se:qf_algebras}

Any multiplicative quadratic form $ \bmu : C_2^N \rightarrow C_2 $
determines the corresponding bicharacter $\beta_{\bmu}$
and this may be considered as a bicharacter
with values in $\mathbb{R}^{\times}$,
taking only the values $\pm 1$.
Hence $\bmu$ determines the associative algebra
$ \mathcal{A}_{\mathbb{R}} ( N , \underline{2} , \underline{\mu} , \beta_{\bmu}) $,
with $ \underline{\mu} = \bigl( \bmu(\mathbf{e}_1) ,
\ldots , \bmu(\mathbf{e}_N) \bigr) $,
where we denote too by $\mathbf{e}_i$ the canonical generators of $C_2^N$
that corresponds to the canonical basic elements of $\mathbb{F}_2^N$ above:
\[
\mathbf{e}_i = ( 1 , \ldots , 1 , \overbrace{-1}^i , 1 , \ldots , 1 ) .
\]

Theorem \ref{th:twisted} has the following direct consequence:

\begin{proposition}\label{pr:mu_twisted}
Given a multiplicative quadratic form $ \bmu : C_2^N \rightarrow C_2 $,
the $C_2^N$-graded algebra
$ \mathcal{A}_{\mathbb{R}} ( N , \underline{2} , \underline{\mu} , \beta_{\bmu}) $
is graded isomorphic to the twisted group algebra $ \mathbb{R}^{\sigma} C_2^N $,
where the $2$-cocycle $\sigma$ is given by Equation \eqref{eq:sigma}.
In particular,
\begin{equation}\label{eq:SigmaMu}
\sigma(q,q) = \bmu(q) \text{ } \forall q \in C_2^N ;
\end{equation}
and
\begin{equation}\label{eq:SigmaBeta}
\sigma( \mathbf{e}_i , \mathbf{e}_j ) =
\begin{cases}
	1 &
	\text{if } i<j , \\
	\beta_{\bmu} ( \mathbf{e}_i , \mathbf{e}_j ) &
	\text{if } i>j .
\end{cases}
\end{equation}
\end{proposition}

\begin{remark}
The graded algebra
$ \mathcal{A}_{\mathbb{R}} ( N , \underline{2} , \underline{\mu} , \beta_{\bmu}) $
is determined by Equations \eqref{eq:beta_comm} and \eqref{eq:SigmaMu},
so it only depends on the quadratic form $\bmu$
and not on the chosen basis of the grading group $C_2^N$,
and we will denote it simply by $\mathcal{A}_{\mathbb{R}}(\bmu)$.
Also, if two multiplicative quadratic forms $\bmu$ and $\bmu'$ are equivalent,
then $\mathcal{A}_{\mathbb{R}}(\bmu)$ is isomorphic to
$\mathcal{A}_{\mathbb{R}}(\bmu')$;
and for any two multiplicative quadratic forms $\bmu$ and $\bmu'$,
$ \mathcal{A}_{\mathbb{R}}(\bmu) \otimes_\mathbb{R}
\mathcal{A}_{\mathbb{R}}(\bmu') $
is isomorphic to $ \mathcal{A}_{\mathbb{R}}( \bmu \perp \bmu' ) $.
\end{remark}

The arguments in \cite{Ovs16} or \cite[Section 5]{BKR17} give:

\begin{theorem}\label{th:isos}
We have the following isomorphisms,
\begin{itemize}
\item $ M_{2^n}(\mathbb{R})
\cong \mathcal{A}_{\mathbb{R}}(\bmu_0^{2n}) $,
\item $ M_{2^{n-1}}(\mathbb{H})
\cong \mathcal{A}_{\mathbb{R}}(\bmu_1^{2n}) $,
\item $ M_{2^n}(\mathbb{R}) \times M_{2^n}(\mathbb{R})
\cong \mathcal{A}_{\mathbb{R}}(\bmu_0^{2n+1}) $,
\item $ M_{2^{n-1}}(\mathbb{H}) \times M_{2^{n-1}}(\mathbb{H})
\cong \mathcal{A}_{\mathbb{R}}(\bmu_1^{2n+1}) $,
\item $ M_{2^n}(\mathbb{C})
\cong \mathcal{A}_{\mathbb{R}}(\bmu_2^{2n+1}) $.
\end{itemize}
\end{theorem}

Therefore, the isomorphism class of the algebra $\mathcal{A}_{\mathbb{R}}(\bmu)$,
for a regular multiplicative quadratic form $\bmu$
is determined by the equivalence class of $\bmu$,
and hence by the dimension $N$ and by the Arf invariant of $\bmu$.
Thus, we can rewrite the previous result as follows:

\begin{corollary}\label{co:isos}
Let $ \bmu : C_2^N \rightarrow C_2 $ be
a regular multiplicative quadratic form.
Then there are the following possibilities:
\begin{itemize}
\item if $N$ is even, $N=2n$, and $\Arf(\bmu)=1$,
then $ \mathcal{A}_{\mathbb{R}}(\bmu) \cong M_{2^n}(\mathbb{R}) $,
\item if $N$ is even, $N=2n$, and $\Arf(\bmu)=-1$,
then $ \mathcal{A}_{\mathbb{R}}(\bmu) \cong M_{2^{n-1}}(\mathbb{H}) $,
\item if $N$ is odd, $N=2n+1$, and $\Arf(\bmu)=1$,
then $ \mathcal{A}_{\mathbb{R}}(\bmu) \cong M_{2^n}(\mathbb{R}) \times M_{2^n}(\mathbb{R}) $,
\item if $N$ is odd, $N=2n+1$, and $\Arf(\bmu)=-1$,
then $ \mathcal{A}_{\mathbb{R}}(\bmu) \cong M_{2^{n-1}}(\mathbb{H}) \times M_{2^{n-1}}(\mathbb{H}) $,
\item if $N$ is odd, $N=2n+1$, and $\Arf(\bmu)=0$,
then $ \mathcal{A}_{\mathbb{R}}(\bmu) \cong M_{2^n}(\mathbb{C}) $.
\end{itemize}
\end{corollary}

\begin{remark}
In \cite[Section 5]{BKR17} (see also \cite{Rod16}) the opposite reasoning was made.
Real graded-division algebras with homogeneous components of dimension $1$
and center of dimension at most $2$ were classified,
obtaining Corollary \ref{co:isos}.
Then, from that classification of gradings,
Dickson's classification of quadratic forms was derived.
\end{remark}

Recall that $\bmu_{p,q}$ denotes the multiplicative version of $\mathbf{q}_{p,q}$.
Equation \eqref{eq:Cl_generators_relations} proves our next result:

\begin{theorem}\label{th:Cl_pq}
The Clifford algebra $\Cl_{p,q}(\mathbb{R})$ is isomorphic to
$\mathcal{A}_{\mathbb{R}}(\bmu_{p,q})$.
\end{theorem}

\section{Real Clifford algebras}\label{se:real_Clifford}

In order to classify the Clifford algebras $\Cl_{p,q}(\mathbb{R})$
we must compute the Arf invariant $\Arf(\bmu_{p,q})$ of
the (multiplicative versions of the) quadratic forms in Example \ref{ex:mu_pq}.
By $\mathrm{sign}(x)$ we denote the sign $+1$, $0$, or $-1$,
of the real number $x$, this sign being $0$ for $x=0$.

\begin{theorem}\label{th:Arf_mupq}
For any $p,q\in\mathbb{N}\cup\{0\}$, the Arf invariant of $\bmu_{p,q}$ is:
\begin{align}
\Arf(\bmu_{p,q})
& = \mathrm{sign} \left( \cos \frac{ (p-q) \pi }{4}
+ \sin \frac{ (p-q) \pi }{4} \right)
\\ & = \begin{cases}
	1 & \text{if } p-q+1 \equiv 1,2,3 \pmod{8} , \\
	0 & \text{if } p-q+1 \equiv 0,4 \pmod{8} , \\
	-1 & \text{if } p-q+1 \equiv 5,6,7 \pmod{8} .
\end{cases}
\end{align}
\end{theorem}

\begin{proof}
Several steps will be followed.

\medskip

\noindent$\bullet$\quad
Let us prove first the formula
\begin{equation}
\Arf(\bmu_{p+1,q+1}) = \Arf(\bmu_{p,q}) .
\end{equation}
To do this, denote by $W$ the subgroup of $C_2^{(p+1)+(q+1)}$ generated by
$ \mathbf{e}_1 , \allowbreak \ldots ,
\allowbreak \mathbf{e}_p , \allowbreak \mathbf{e}_{p+2} ,
\allowbreak \ldots , \allowbreak \mathbf{e}_{p+q+1} $.
Then for any $ w \in W $,
$ \bmu( w \mathbf{e}_{p+1} ) \neq \bmu( w \mathbf{e}_{p+q+2} ) $,
because $ \beta_{\bmu} ( w , \mathbf{e}_{p+1} )
= \beta_{\bmu} ( w , \mathbf{e}_{p+q+2} ) $,
and $ \bmu(\mathbf{e}_{p+1}) = 1 $
and $ \bmu(\mathbf{e}_{p+q+2}) = -1 $.
On the other hand, for any $ w \in W $,
$ \bmu(w) = \bmu( w \mathbf{e}_{p+1} \mathbf{e}_{p+q+2} ) $,
because $ \beta_{\bmu} ( w , \mathbf{e}_{p+1} \mathbf{e}_{p+q+2} ) = 1 $
and $ \bmu( \mathbf{e}_{p+1} \mathbf{e}_{p+q+2} ) = 1 $.
Since the group $C_2^{(p+1)+(q+1)}$ can be expressed as the disjoint union
$ W \cup W \mathbf{e}_{p+1} \cup W \mathbf{e}_{p+q+2}
\cup W \mathbf{e}_{p+1} \mathbf{e}_{p+q+2} $, we get
\begin{align*}
\big\lvert \{ w \in C_2^{(p+1)+(q+1)}
& : \bmu_{p+1,q+1}(w) = 1 \} \big\rvert = \\
& = 2 \bigl\lvert \{ w \in W
: \bmu_{p+1,q+1}(w) = 1 \} \big\rvert + \lvert W \rvert , \\
\big\lvert \{ w \in C_2^{(p+1)+(q+1)}
& : \bmu_{p+1,q+1}(w) = -1 \} \big\rvert = \\
& = 2 \bigl\lvert \{ w \in W
: \bmu_{p+1,q+1}(w) = -1 \} \big\rvert + \lvert W \rvert ,
\end{align*}
and the result follows.

\medskip

\noindent$\bullet$\quad
Therefore we get
\begin{equation}\label{eq:reduction_Arf}
\Arf(\bmu_{p,q}) = \begin{cases}
	\Arf(\bmu_{1,1}) & \text{if } p=q , \\
	\Arf(\bmu_{p-q,0}) & \text{if } p>q , \\
	\Arf(\bmu_{0,q-p}) & \text{if } p<q ,
\end{cases}
\end{equation}
and it is enough to check our formulas for
$\bmu_{N,0}$ and $\bmu_{0,N}$, $ N \in \mathbb{N} $.

For $\bmu_{N,0}$, consider the canonical generators
$\mathbf{e}_i$ of $C_2^N$ as in Example \ref{ex:mu_pq}.
For any sequence $ I = \{ 1 \leq i_1 < i_2 < \cdots < i_r \leq N \} $,
let $ \mathbf{e}_I \bydef \mathbf{e}_{i_1} \mathbf{e}_{i_2} \cdots \mathbf{e}_{i_r} $.
Then $ \bmu_{N,0}(\mathbf{e}_I) = (-1)^{\binom{r}{2}} $,
so we arrive at
\begin{equation}\label{eq:mu_N0}
\vert \bmu_{N,0}^{-1} (+1) \vert -
\vert \bmu_{N,0}^{-1} (-1) \vert
= \sum_{r=0}^{N} \binom{N}{r} (-1)^{(r-1)r/2}
= S_0 + S_1 - S_2 - S_3 .
\end{equation}
While for $\bmu_{0,N}$,
we have $ \bmu_{0,N}(\mathbf{e}_I)
= (-1)^r (-1)^{\binom{r}{2}} = (-1)^{r+\binom{r}{2}} $, so
\begin{equation}\label{eq:mu_0N}
\vert \bmu_{0,N}^{-1} (+1) \vert -
\vert \bmu_{0,N}^{-1} (-1) \vert
= \sum_{r=0}^{N} \binom{N}{r} (-1)^{(r-1)r/2+r}
= S_0 - S_1 - S_2 + S_3 .
\end{equation}
Where $S_0$, $S_1$, $S_2$ and $S_3$ are
the following binomial sums:
\begin{equation}\label{eq:bin_sum_0}
S_0 := \binom{N}{0} + \binom{N}{4} + \dots
+ \binom{N}{ 4 \lfloor \frac{N-0}{4} \rfloor + 0 }
= \frac{1}{2} \left( 2^{N-1} + 2^{N/2}
\cos \frac{N\pi}{4} \right) ,
\end{equation}
\begin{equation}\label{eq:bin_sum_1}
S_1 := \binom{N}{1} + \binom{N}{5} + \dots
+ \binom{N}{ 4 \lfloor \frac{N-1}{4} \rfloor + 1 }
= \frac{1}{2} \left( 2^{N-1} + 2^{N/2}
\sin \frac{N\pi}{4} \right) ,
\end{equation}
\begin{equation}\label{eq:bin_sum_2}
S_2 := \binom{N}{2} + \binom{N}{6} + \dots
+ \binom{N}{ 4 \lfloor \frac{N-2}{4} \rfloor + 2 }
= \frac{1}{2} \left( 2^{N-1} - 2^{N/2}
\cos \frac{N\pi}{4} \right) ,
\end{equation}
\begin{equation}\label{eq:bin_sum_3}
S_3 := \binom{N}{3} + \binom{N}{7} + \dots
+ \binom{N}{ 4 \lfloor \frac{N-3}{4} \rfloor + 3 }
= \frac{1}{2} \left( 2^{N-1} - 2^{N/2}
\sin \frac{N\pi}{4} \right) .
\end{equation}
The last equality for each binomial sum is well-known,
but we give an algebraic proof of them below,
in the last step of this proof of Theorem \ref{th:Arf_mupq}.

Taking into account Equations \eqref{eq:bin_sum_0}-\eqref{eq:bin_sum_3}, we get
\begin{equation*}
\Arf(\bmu_{N,0}) = \mathrm{sign} \left(
\cos \frac{N\pi}{4} +
\sin \frac{N\pi}{4} \right) ,
\end{equation*}
and
\begin{equation*}
\Arf(\bmu_{0,N}) = \mathrm{sign} \left(
\cos \frac{N\pi}{4} -
\sin \frac{N\pi}{4} \right) ,
\end{equation*}
which, together with Equation \eqref{eq:reduction_Arf}, give
\begin{equation*}
\Arf(\bmu_{p,q}) = \mathrm{sign} \left(
\cos \frac{(p-q)\pi}{4} +
\sin \frac{(p-q)\pi}{4} \right)\, .
\end{equation*}
This proves the first half of the formula in the statement of the Theorem.
The second half is clear.

\medskip

\noindent$\bullet$\quad
To prove Equations \eqref{eq:bin_sum_0}-\eqref{eq:bin_sum_3},
consider the isomorphism of
unital $\mathbb{R}$-algebras given by
the Chinese Remainder Theorem:
\begin{align}
\varphi : \mathbb{R}[T] / (T^4-1)
& \longrightarrow \mathbb{R}[X] / (X-1)
\times \mathbb{R}[Y] / (Y+1)
\times \mathbb{R}[Z] / (Z^2+1)
\\ f(T)+(T^4-1)
& \longmapsto ( f(X)+(X-1)
, f(Y)+(Y+1) , f(Z)+(Z^2+1) ) .
\nonumber
\end{align}
If we write $ t = T+(T^4-1) $
and $ i = Z+(Z^2+1) $,
then $\varphi$ is defined by
$ \varphi(1) = (1,1,1) $,
$ \varphi(t) = (1,-1,i) $,
$ \varphi(t^2) = (1,1,-1) $ and
$ \varphi(t^3) = (1,-1,-i) $.
Reciprocally, $\varphi^{-1}$ is determined by
$ \varphi^{-1}(1,0,0) = (1+t+t^2+t^3)/4 $,
$ \varphi^{-1}(0,1,0) = (1-t+t^2-t^3)/4 $,
$ \varphi^{-1}(0,0,1) = (1-t^2)/2 $ and
$ \varphi^{-1}(0,0,i) = (t-t^3)/2 $.

On the one hand,
we can compute $(1+t)^N$ in $\mathbb{R}[t]$,
\begin{equation}
(1+t)^N = \sum_{i=0}^N \binom{N}{i} t^i
= S_0 1 + S_1 t + S_2 t^2 + S_3 t^3 .
\end{equation}
On the other hand,
we can apply first $\varphi$ to $(1+t)$,
which gives $(2,0,1+i)$,
and then raise it to the $N$-th power,
\begin{equation}
(2,0,1+i)^N = (2^N,0,(1+i)^N) = (2^N,0,
2^{N/2} \cos \frac{N\pi}{4} +
2^{N/2} \sin \frac{N\pi}{4} i ) .
\end{equation}
Applying $\varphi^{-1}$ to this last expression,
and comparing it with the other one we have,
we get our desired formulas.
\end{proof}

Corollary \ref{co:isos}, together with
Theorems \ref{th:Cl_pq} and \ref{th:Arf_mupq},
gives immediately alternative proofs of the following
classical results on Clifford algebras over the real numbers:

\begin{corollary}\label{cor:Clifford_isos}
Let $ p,q \in \mathbb{N} \cup \{ 0 \} $ and $N=p+q$.
Then,
\begin{description}
\item[$N$ even] $ $
\begin{itemize}
	\item $ \Cl_{p,q}(\mathbb{R}) \cong M_{2^{N/2}}(\mathbb{R}) $
	if $ p-q+1 \equiv 1,3 \pmod{8} $.
	\item $ \Cl_{p,q}(\mathbb{R}) \cong M_{2^{(N-2)/2}}(\mathbb{H}) $
	if $ p-q+1 \equiv 5,7 \pmod{8} $.
\end{itemize}
\item[$N$ odd] $ $
\begin{itemize}
	\item $ \Cl_{p,q}(\mathbb{R}) \cong M_{2^{(N-1)/2}}(\mathbb{R}) \times M_{2^{(N-1)/2}}(\mathbb{R}) $
	if $ p-q+1 \equiv 2 \pmod{8} $.
	\item $ \Cl_{p,q}(\mathbb{R}) \cong M_{2^{(N-1)/2}}(\mathbb{C})$
	if $ p-q+1 \equiv 0,4 \pmod{8} $.
	\item $ \Cl_{p,q}(\mathbb{R}) \cong M_{2^{(N-3)/2}}(\mathbb{H}) \times M_{2^{(N-3)/2}}(\mathbb{H}) $
	if $ p-q+1 \equiv 6 \pmod{8} $.
\end{itemize}
\end{description}
\end{corollary}

\begin{corollary}\label{cor:periodicity}
Let $ p,q \in \mathbb{N} \cup \{ 0 \} $,
then the following `periodicity results' hold:
\begin{enumerate}
\item $ \Cl_{p+1,q+1}(\mathbb{R}) \cong
\Cl_{p,q}(\mathbb{R}) \otimes_{\mathbb{R}} M_2(\mathbb{R}) $.
\item $ \Cl_{p+2,q}(\mathbb{R}) \cong
\Cl_{q,p}(\mathbb{R}) \otimes_{\mathbb{R}} M_2(\mathbb{R}) $.
\item $ \Cl_{p,q+2}(\mathbb{R}) \cong
\Cl_{q,p}(\mathbb{R}) \otimes_{\mathbb{R}} \mathbb{H} $.
\item $ \Cl_{p+4,q}(\mathbb{R}) \cong
\Cl_{p,q+4}(\mathbb{R}) $.
\end{enumerate}
\end{corollary}

\begin{proof}
For the first case, it is enough to show that
$ \Arf(\bmu_{p+1,q+1}) = \Arf(\bmu_{p,q}) $,
and this was proved in the proof of Theorem \ref{th:Arf_mupq}.
For the second case, note that $ q-p+1 = 4 - \bigl( (p+2)-q+1 \bigr) $,
so $ \Arf(\bmu_{p+2,q}) = \Arf(\bmu_{q,p}) $ by Theorem \ref{th:Arf_mupq}.
Also, $ q-p+1 = - \bigl( p-(q+2)+1 \bigr) $,
so Theorem \ref{th:Arf_mupq} shows that
$ \Arf(\bmu_{p,q+2}) = -\Arf(\bmu_{q,p}) $,
and this proves the third case.
Finally, $ p+4-q+1 \equiv p-(q+4)+1 \pmod{8} $,
so $ \Arf(\bmu_{p+4,q}) = \Arf(\bmu_{p,q+4}) $,
and this proves the last case.
\end{proof}

Alternatively, we can prove that
\begin{equation}
\Arf(\bmu_{p+2,q}) = \Arf(\bmu_{q,p})
\end{equation}
reasoning in a similar way as we did in the proof of Theorem \ref{th:Arf_mupq}.
Let $ \mathbf{e}_1 , \dots , \mathbf{e}_N $ and
$ \mathbf{f}_1 , \dots , \mathbf{f}_N , \mathbf{f}_{N+1} , \mathbf{f}_{N+2} $
be the generators of $C_2^{q+p}$ and $C_2^{(p+2)+q}$ as in Example \ref{ex:mu_pq},
but reorder the indexes so that
$ \bmu_{p+2,q} (\mathbf{f}_{N+1}) = \bmu_{p+2,q} (\mathbf{f}_{N+2}) = +1 $
and $ \bmu_{q,p} (\mathbf{e}_i) = - \bmu_{p+2,q} (\mathbf{f}_i) $ for $ i = 1 , \dots , N $.
We have again four types of terms in $C_2^{(p+2)+q}$:
(1) the $ \mathbf{f}_I $,
(2) the $ \mathbf{f}_I \mathbf{f}_{N+1} \mathbf{f}_{N+2} $,
(3) the $ \mathbf{f}_I \mathbf{f}_{N+1} $,
(4) the $ \mathbf{f}_I \mathbf{f}_{N+2} $
($ I = \{ 1 \leq i_1 < i_2 < \cdots < i_r \leq N \} $,
$ \mathbf{f}_I = \mathbf{f}_{i_1} \mathbf{f}_{i_2} \cdots \mathbf{f}_{i_r} $).
Since $ \bmu_{p+2,q} ( \mathbf{f}_{N+1} \mathbf{f}_{N+2} ) = -1 $
and $ \beta_{\bmu} (\mathbf{f}_I , \mathbf{f}_{N+1} \mathbf{f}_{N+2} ) = 1 $,
the terms of the second type cancel out
with the terms of the first type,
$ \bmu_{p+2,q} (\mathbf{f}_I) =
- \bmu_{p+2,q} ( \mathbf{f}_I \mathbf{f}_{N+1} \mathbf{f}_{N+2} ) $.
Also, we have that $ \bmu_{p+2,q} ( \mathbf{f}_I \mathbf{f}_{N+1} )
= \bmu_{p+2,q} ( \mathbf{f}_I \mathbf{f}_{N+2} ) $,
hence the contribution of the terms of the fourth type to $\Arf(\bmu_{p+2,q})$
is the same as that of the terms of the third type.
If $ \vert I \vert $ is odd, then
\[
\bmu_{p+2,q} ( \mathbf{f}_I \mathbf{f}_{N+1} )
= - \bmu_{p+2,q}(\mathbf{f}_I) \bmu_{p+2,q}(\mathbf{f}_{N+1})
= \bmu_{q,p}(\mathbf{e}_I) ;
\]
whereas if $ \vert I \vert $ is even, also
\[
\bmu_{p+2,q} ( \mathbf{f}_I \mathbf{f}_{N+1} )
= \bmu_{p+2,q}(\mathbf{f}_I) \bmu_{p+2,q}(\mathbf{f}_{N+1})
= \bmu_{q,p}(\mathbf{e}_I) .
\]
We conclude that $ \Arf(\bmu_{p+2,q}) = \Arf(\bmu_{q,p}) $.

The same argument works for $\bmu_{p,q+2}$ and $\bmu_{q,p}$,
but this time we would have $ \bmu_{p,q+2} (\mathbf{f}_{N+1})
= \bmu_{p,q+2} (\mathbf{f}_{N+2}) = -1 $,
so we get
\begin{equation}
\Arf(\bmu_{p,q+2}) = -\Arf(\bmu_{q,p}) .
\end{equation}

\section*{Acknowledgements}

Both authors were supported by the Spanish
Mi\-nis\-te\-rio de Econom\'ia y Competitividad
and Fondo Europeo de De\-sa\-rro\-llo Regional FEDER
(MTM2013-45588-C3-2-P).
Adri\'an Rodrigo-Escudero was also supported by
a doctoral grant of the Diputaci\'on General de Arag\'on.

\end{document}